\RequirePackage{etex}
\RequirePackage{etex}
\RequirePackage{easymat}
\documentclass[12pt, a4paper]{elsarticle}
\usepackage{amsthm,amsmath,amssymb,extsizes,easymat,graphicx,hyperref,easybmat}
\usepackage{tikz}
\usepackage{easybmat}
\usepackage{hyperref}
\usepackage[all]{xy}
\usepackage[active]{srcltx}

\usepackage{amsmath,amssymb,amsthm,mathtools}

\usepackage[active]{srcltx}
\usepackage{MnSymbol}
\usepackage{mathrsfs}
\DeclareMathOperator{\orb}{O}
\DeclareMathOperator{\nrk}{nrk}

\DeclareRobustCommand{\coin}{%
  \tikz[baseline=-0.5ex]{%
    \draw[color=blue] (0,0) circle (0.1cm);
    \draw[color=blue] (-0.1,0) -- (0.1,0);
    \draw[color=blue] (0,-0.1) -- (0,0.1);
  }%
}

\newcommand{\sdotsss}%
{\text{\raisebox{-2.2pt}{$\cdot\,$}%
 \raisebox{1.7pt}{$\cdot$}%
\raisebox{5.6pt}{$\,\cdot$}}}

\renewcommand{\le}{\leqslant}
\renewcommand{\ge}{\geqslant}

\newtheorem{theorem}{Theorem}[section]

\newtheorem{lemma}[theorem]{Lemma}

\newtheorem{remark}[theorem]{Remark}

\newtheorem{example}[theorem]{Example}

\usepackage{tcolorbox}
\tcbuselibrary{skins}
\tcbuselibrary{breakable}
\tcbset{shield externalize}

\newcommand{\hide}[1]{}
\makeatletter
\def\ps@pprintTitle{%
   \let\@oddhead\@empty
   \let\@evenhead\@empty
   \let\@oddfoot\@empty
   \let\@evenfoot\@empty
}
\makeatother
\begin{document}

\title{Minimal degenerations of orbits of skew-symmetric matrix pencils}

\tnotetext[t1]{The work was supported by the Swedish Research Council (VR) under grant 2021-05393. The authors report there are no competing interests to declare.}

\date{}

\author[um]{Sweta Das}
\ead{sweta.das@oru.se}

\author[vn,um]{Andrii Dmytryshyn}
\ead{andrii@chalmers.se}

\address[um]{School of Science and Technology, Örebro University, 70182 Örebro, Sweden}
\address[vn]{Department of Mathematical Sciences, Chalmers University of Technology and University of Gothenburg, 41296 Gothenburg, Sweden}

\begin{abstract}
Complete eigenstructure, e.g., eigenvalues with multiplicities and minimal indices, of a skew-symmetric matrix pencil may change drastically if the matrix coefficients of the pencil are subjected to (even small) perturbations. These changes can be investigated qualitatively by constructing the stratification (closure hierarchy) graphs of the congruence orbits of the pencils. The results of this paper facilitate the construction of such graphs by providing all closest neighbours for a given node in the graph. More precisely, we prove a necessary and sufficient condition for one congruence orbit of a skew-symmetric matrix pencil, $A$, to belong to the closure of the congruence orbit of another pencil, $B$, such that there is no pencil, $C$, whose orbit contains the closure of the orbit of $A$ and is contained in the closure of the orbit of $B$.
\end{abstract}

\begin{keyword} matrix pencil\sep congruence\sep skew-symmetry\sep stratification\sep eigenstructure\sep canonical form

\MSC 15A22\sep 15A21\sep 15A18
\end{keyword}

\maketitle

\section{Introduction}
\label{introd} 
Complete eigenstructure, e.g., eigenvalues with multiplicities and minimal indices, of a matrix pencil is essential for understanding physical properties of the system described by this pencil. Unfortunately, small perturbations in the coefficients of the pencil may cause significant changes in the eigenstructure. This sensitivity challenges the accuracy of analysis and computations that are dependent on the eigenstructure of the input matrix pencil. One way to study these changes is by constructing the closure hierarchy graphs (or stratifications) for orbits and bundles of matrix pencils \cite{Bole98,DmJK17,DmKa14,FeGP00,Garc98,Pokr86}. 
Stratification provides us with qualitative information about the eigenstructure of the matrix pencils in a neighbourhood of a given matrix pencil. More precisely, stratification of orbits of matrix pencils is a graph in which each node represents an orbit and each edge represents a closure relation, i.e., there is an edge between two nodes if the closure of the orbit corresponding to one node contains the closure of the orbit corresponding to another node, see e.g., \cite{EdEK99,Stra24}.  

In many applications, the matrix coefficients of the pencils possess a certain structure, e.g., they may be (skew-)symmetric, (skew-)Hermitian, palindromic, or have a particular blocking structure. Preserving these structures in the analysis and computations may be crucial for applications \cite{MMMM06b}. Therefore, we aim to develop the stratification theory for the structured matrix pencils, which may differ significantly from the theory for general  matrix pencils, see e.g., \cite{DeDD20a,DeDD24}. In particular, this paper contributes to the development of such a theory  for skew-symmetric matrix pencils. Recall that $A - \lambda B$ is skew-symmetric if and only if $A = -A^T$ and $B = -B^T$.  
Skew-symmetric matrix pencils have various applications in mathematics, engineering and physics. Notable ones are in control theory \cite{BoMe06, KaPK, LiHo99}, systems with bi-hamiltonian structure \cite{Olve91} and multisymplectic partial differential equations \cite{BrRe01}.

Skew-symmetric matrix pencils have received a considerable attention in the recent years, see e.g., papers on canonical forms \cite{Rodm07, Thom91}, structured staircase algorithm \cite{BrMe07, ByMX07}, computation of codimensions of the congruence orbits and bundles \cite{DmJK13,DmKS13}, miniversal deformations and perturbation theories \cite{Dmyt16,DoPV19}, stratification theory \cite{DmKa14}, and problems for skew-symmetric linearizations of skew-symmetric matrix polynomials \cite{DeDD24a, Dmyt17, DmDo18}. 
Moreover, in \cite{DmKa14} an algorithm for stratification of skew-symmetric matrix pencils is presented. Nevertheless, it is computationally demanding even for the pencils of small sizes and the amount of required computations grows rapidly when the size of pencils increases. Such a high computational cost arises partially from the use of the stratification of general matrix pencils under strict equivalence from which the algorithm extracts the nodes corresponding to the skew-symmetrizable matrix pencils and checks for the paths between these nodes. Using this extracted information a new stratification graph for skew-symmetric matrix pencils under congruence is then constructed. For example, to draw the orbit stratification graph of $7 \times 7$ skew-symmetric matrix pencil under congruence (presented in Figure \ref{fig:Orbit_Ex_7}), one needs to consider the stratification graph of a $7 \times 7$ general matrix pencils, consisting of 570 nodes and 2471 edges (these exact numbers are taken from the Stratigraph software \cite{Stra24}), identify and extract the 20 relevant nodes and check the paths between these nodes in the graph for general matrix pencils. 
%
%
Using the results of this paper, one can draw the stratification graph of skew-symmetric matrix pencils without a need to consider the stratification graph for general matrix pencils, see Examples \ref{orbit_ex} and \ref{example2}. 

This paper closes the gap in the stratification theory of orbits of skew-symmetric matrix pencils. Namely, it explains when two orbits are neighbours in the stratification graph, which allows us to construct such graphs efficiently. We also emphasize that as with the general matrix pencils, this result corresponds to the description of the minimal changes that may occur to the eigenstructure of a skew-symmetric matrix pencil under arbitrarily small skew-symmetric perturbation. The corresponding results are already known for general matrix pencils \cite{EdEK99} and matrix polynomials \cite{DJKV20}, as well as for controllablity and observability pairs \cite{ElJK09}, and system pencils \cite{DmJK17}.   




In Section 2, we recall definitions, notations, and results related to the orbit stratification of matrix pencils along with some observations about skew-symmetric matrix pencils. In Section 3, we present our main results about necessary and sufficient conditions for two orbits of skew-symmetric matrix pencils to be neighbours in the closure hierarchy graphs and illustrate our results with examples. 
Section 4 presents a relation between closure hierarchy graphs of matrix pencils of different dimensions. Conclusions and future plans are presented in Section 5. 

\section{Preliminary results}
\label{prel}

Define $\mathbb{Q}_+$, $\mathbb{Z}_+$, and $\mathbb{C}$ as the sets of positive rational numbers, positive integers and complex numbers, respectively.  All matrices considered in this manuscript are complex matrices. For each $k=1,2, \ldots $ define the $k\times k$
matrices
\begin{equation*}\label{1aa}
J_k(\mu):=\begin{bmatrix}
\mu&1&&\\
&\mu&\ddots&\\
&&\ddots&1\\
&&&\mu
\end{bmatrix},\qquad
I_k:=\begin{bmatrix}
1&&&\\
&1&&\\
&&\ddots&\\
&&&1
\end{bmatrix},
\end{equation*}
and for each $k=0,1, \ldots $ define the $k\times
(k+1)$ matrices
\begin{equation*}
F_k :=
\begin{bmatrix}
0&1&&\\
&\ddots&\ddots&\\
&&0&1\\
\end{bmatrix}, \qquad
G_k :=
\begin{bmatrix}
1&0&&\\
&\ddots&\ddots&\\
&&1&0\\
\end{bmatrix}.
\end{equation*}
All non-specified entries of $J_k(\mu), I_k, F_k,$ and $G_k$ are zeros.

\subsection{Matrix pencils under strict equivalence}

An $m \times n$ matrix pencil $A - \lambda B$ is called strictly equivalent to the pencil $C - \lambda D$ if and only if there are non-singular matrices $Q$ and $R$ such that $(Q^{-1}AR,Q^{-1}BR) = (C,D)$. The {\it orbit} of $A - \lambda B$ under the
action of the group $GL_m(\mathbb C) \times GL_n(\mathbb C)$ on the space of all matrix pencils by strict equivalence is defined as follows:
\begin{equation*} \label{equorbit}
\orb_{A - \lambda B}^e = \{Q^{-1} (A - \lambda B) R \ | \ Q \in GL_m(\mathbb C), R \in GL_n(\mathbb C)\}.
\end{equation*} 

The following theorem gives a canonical form known as the Kronecker canonical form (KCF) for matrix pencils under the strict equivalence transformation.

\begin{theorem}{ \cite{Gant59}}\label{kron}
Each $m \times n$ matrix pencil $A - \lambda B$ is strictly equivalent 
to a direct sum, uniquely determined up
to permutation of summands, of pencils of the form
\begin{align*}
 E_k(\mu)&:=J_k(\mu)- \lambda I_k, \mu \in \mathbb C, \quad E_k(\infty):=I_k- \lambda J_k(0), \\
 L_k&:=F_k- \lambda G_k, \quad \text{ and } \quad L_k^T:=F^T_k - \lambda G_k^T. 
\end{align*}
\end{theorem}
The blocks $E_k(\mu)$, $E_k(\infty)$, $L_k$ and $L_k^T$  correspond to the finite eigenvalues, infinite eigenvalues, column minimal indices and row minimal indices, respectively, of $A - \lambda B$. The blocks $E_k(\mu)$, $E_k(\infty)$ altogether form the regular part and the rest forms the singular part of the matrix pencil. 

An integer partition of $n \in \mathbb{Z}_+$ is a non-increasing sequence of integers ${\cal P}=(p_1, p_2, p_3, \ldots)$ such that $p_i \ge 0$ and the sum of all $p_i$ is $n$. Each of $p_i$'s is called {\it part} of the integer partition. For any $a \in \mathbb Z$ and $b \in \mathbb{Q}_+$, ${\cal P}+a$ and $b{\cal P}$ are defined as addition of $a$ to each part of the partition ${\cal P}$ and multiplication by $b$ of each part of the partition ${\cal P}$, respectively, such that the resulting sequence is a non-increasing sequence of non-negative integers. The set of all integer partitions forms a lattice with respect to the dominance order. By dominance order on integer partitions, we mean that for any integer partitions ${\cal P}$ and ${\cal Q}$,  ${\cal P} \succcurlyeq {\cal Q}$ if and only if $p_{1} + p_{2} + \ldots + p_{m} \ge q_{1} + q_{2} + \ldots + q_{m}$, for $m = 1,2,$ \ldots, and ${\cal P} \succ {\cal Q}$ if and only if ${\cal P} \succcurlyeq {\cal Q}$ and ${\cal P} \neq {\cal Q}$. In a lattice, we say ${\cal P}$ covers ${\cal Q}$ if and only if ${\cal P} \succ {\cal Q}$ and there exists no partition $\mathcal{Z}$ that $\mathcal{Z} \succ {\cal Q}$ and  ${\cal P} \succ \mathcal{Z}$. The cover relation between two integer partitions can also be illustrated in the following way: \\
Place $m$ coins in a table in a non-increasing sequence with $m_{i}$ coins in column $i$. An integer partition ${\cal P}_{1}$ covers ${\cal P}_{2}$ if ${\cal P}_{2}$ may be obtained from ${\cal P}_{1}$ by moving a coin downward one row or rightward one column such that the resulting integer partition remains monotonic. Such coin move will be referred to as {\it{minimum rightward coin move}}. Equivalently, {\it{minimum leftward coin move}} is referred to a coin move upward one row or leftward one column to obtain ${\cal P}_{1}$ from ${\cal P}_{2}$ (for more information, see e.g., \cite{EdEK99}).

We define $\overline{\mathbb{C}} = \mathbb{C} \cup \infty$. The following integer partitions, known as the Weyr characteristics, are associated with every matrix pencil $P$ (with eigenvalues $\mu_j \in \overline{\mathbb C}$):
\begin{itemize}
    \item ${\cal R}(P)=(r_0(P),r_1(P), \dots )$: the $k^{th}$ position is the number of $L$-blocks with indices greater than or equal to $k$ (the position numeration starting from 0);
    \item ${\cal L}(P)=(l_0(P),l_1(P), \dots )$: the $k^{th}$ position is the number of $L^T$-blocks with indices greater than or equal to $k$ (the position numeration starting from 0);
    \item $\{ {\cal J}_{\mu_j}(P): j=1, \dots, d \},$ where $d$ is the number of distinct eigenvalues: for each distinct $\mu_j$, ${\cal J}_{\mu_j}(P)=(j_1^{\mu_j}(P),j_2^{\mu_j}(P), \dots),$  the $k^{th}$  position is the number of Jordan blocks of $P$ of size greater than or equal to $k$ (the position numeration starting from 1).
\end{itemize}

For an $m\times n$ matrix pencil $P$, the normal rank of $P$ is defined as the rank of the matrix $P$ over $\mathbb{C}(\lambda)$. It can also be formulated as $nrk(P) = n - r_0(P) = m - l_0(P)$, see \cite{EdEK99}.

By $\overline{ \cal X }$ we denote the closure of a set ${\cal X}$ in the Euclidean topology. 
For two matrix pencils X and Y, by the structure transition $X \rightsquigarrow Y$ we mean that in the canonical form of a matrix pencil the blocks represented by $X$ are replaced by the blocks $Y$. Note that $X$ and $Y$ must have the same sizes.
The following theorem helps us to understand all possible changes in the KCF of matrix pencils under small perturbations.
\begin{theorem}\cite{Bole98, DmKa14, Pokr86}\label{equi-struc-trans}
Let $P_1$ and $P_2$ be two matrix pencils such that $\orb_{P_2}^e \subset \overline {\orb_{P_1}^e}$.
Then $P_2$ can be obtained from $P_1$ changing canonical blocks of $P_1$ by applying a sequence of structure transitions and each transition is one of the six types below:
\begin{enumerate}
    \item[G1.] $L_{j} \oplus L_{k} \rightsquigarrow L_{j-1} \oplus L_{k+1}$, $1\le j \le k;$
    \item[G2.] $L_{j}^T \oplus L_{k}^T \rightsquigarrow L_{j-1}^T \oplus L_{k+1}^T$, $1\le j \le k;$
    \item[G3.] $L_{j+1} \oplus E_{k}(\mu) \rightsquigarrow L_{j} \oplus E_{k+1}(\mu)$, $j,k=0,1,2, \dots$ and $\mu \in \overline{\mathbb C};$ 
    \item[G4.] $L_{j+1}^T \oplus E_{k}(\mu) \rightsquigarrow L_{j}^T \oplus E_{k+1}(\mu)$, $j,k=0,1,2, \dots$ and $\mu \in \overline{\mathbb C};$ 
    \item[G5.] $E_{j-1}(\mu) \oplus E_{k+1}(\mu) \rightsquigarrow E_{j}(\mu) \oplus E_{k}(\mu)$, $1\le j \le k$ and $\mu \in \overline{\mathbb C};$
    \item[G6.] $\bigoplus_{i=1}^tE_{k_i}(\mu_i) \rightsquigarrow L_{p} \oplus L_{q}^T$, if $p+q+1= \sum_{i=1}^t k_i$ and $\mu_i \neq \mu_{i'}$ for $i \neq i', \mu_i \in \overline{\mathbb C}.$ 
\end{enumerate}
\end{theorem}

For two matrix pencils $P_1$ and $P_2$, $P_1 \neq P_2$, $\orb_{P_1}^e \text{ covers } \orb_{P_2}^e$ means $\orb_{P_2}^e \subset \overline {\orb_{P_1}^e}$ and there exist no other skew-symmetric matrix pencil, $P$, $P_1 \neq P$ and $P \neq P_2$ such that $\orb_{P}^e \subset \overline {\orb_{P_1}^e}$ and $\orb_{P_2}^e \subset \overline {\orb_{P}^e}$. Theorem \ref{coin-equi} is a necessary and sufficient condition for orbit cover relation of matrix pencils in terms of coin moves.  
\begin{theorem}\cite{EdEK99}\label{coin-equi}
Let $P_1$ and $P_2$ be two matrix pencils. Then $\orb_{P_1}^e$ covers $\orb_{P_2}^e$ if and only if $P_2$ can be obtained from $P_1$ by applying one of four rules to the integer partition of $P_1$:
\begin{itemize}  
\item[C1.] Minimum rightward coin move from column $j$ in ${\cal R}$ (or ${\cal L}$), where $j\ge 1$;
\item[C2.] If the rightmost column in ${\cal R}$ (or ${\cal L}$) is single coin, append that coin as a new rightmost column of some ${\cal J}_{\mu_i}$;
\item[C3.] Minimum leftward coin move in any ${\cal J}_{\mu_i}$;
\item[C4.] Let k denote the total number of coins in all of the longest (=lowest) rows from all of the ${\cal J}_{\mu_i}$. If the total number of non-zero columns of ${\cal R}$ and ${\cal L}$ exceeds  k+1, remove these $k$ coins, add one more coin to the set, and distribute k+1 coins to $r_p$, $p = 0,\ldots,t$ and $l_q$, $q = 0,\ldots,k-t-1$ such that at least all non-zero columns of ${\cal R}$ and ${\cal L}$ are given coins.
\end{itemize}
\end{theorem}
For convenience, we reformulate Theorem \ref{coin-equi} in terms of structure transitions.
\begin{theorem}\label{struc-equi}
    Let $P_1$ and $P_2$ be two matrix pencils. Then $\orb_{P_1}^e$ covers $\orb_{P_2}^e$ if and only if $P_2$ can be obtained from $P_1$ by applying one of six structure transitions to the canonical blocks of $P_1$:
    \begin{itemize}
        \item[J1.] $L_{j} \oplus L_{k} \rightsquigarrow L_{j-1} \oplus L_{k+1}$, $1\le j \le k$ such that $k - j = \min\big\{\{k_1-j,k_2-j,\dots\}\cup \{k-j_1,k-j_2,\dots\}\big\}$, $j_t\le k$ and $j \le k_u $;
        \item[J2.] $L_{j}^T \oplus L_{k}^T \rightsquigarrow L_{j-1}^T \oplus L_{k+1}^T$, $1\le j \le k$ such that $k - j = \min\big\{\{k_1-j,k_2-j,\dots\}\cup \{k-j_1,k-j_2,\dots\}\big\}$, $j_t\le k$ and $j \le k_u $;
        \item[J3.] $L_{j+1} \oplus E_{k}(\mu) \rightsquigarrow L_{j} \oplus E_{k+1}(\mu)$, $j,k=0,1,2, \dots$ and $\mu \in \overline{\mathbb C}$ such that $k$ and $j+1$ are the sizes of the largest $E$-block and $L$-block in $P_1$, respectively, provided that $r_{j+1}(P_1) = 1$;
        \item[J4.] $L_{j+1}^T \oplus E_{k}(\mu) \rightsquigarrow L_{j}^T \oplus E_{k+1}(\mu)$, $j,k=0,1,2, \dots$ and $\mu \in \overline{\mathbb C}$  such that $k$ and $j+1$ are the sizes of the largest $E$-block and $L^T$-block in $P_1$, respectively, provided that $r_{j+1}(P_1) = 1$;
        \item[J5.] $E_{j-1}(\mu) \oplus E_{k+1}(\mu) \rightsquigarrow E_{j}(\mu) \oplus E_{k}(\mu)$, $\mu \in \overline{\mathbb C}$, $1\le j \le k$ such that $k - j = \min\big\{\{k_1-j,k_2-j,\dots\}\cup \{k-j_1,k-j_2,\dots\}\big\}$, $j_t\le k$ and $j \le k_u $;
        \item[J6.] $\bigoplus_{i=1}^tE_{k_i}(\mu_i) \rightsquigarrow L_{p} \oplus L_{q}^T$, $p\ge p_l$, $q \ge q_j$, for all existing $L_{p_l}$ and $L_{q_j}^T$ in $P_1$; $p+q+1 = \sum_{i=1}^d k_i$, $k_i$ is the size of the largest Jordan block in $P_1$ corresponding to each distinct $\mu_i \in \overline{\mathbb C}.$
    \end{itemize}
\end{theorem}
\begin{proof}
    We explain below the extra requirements added to the rules in Theorem \ref{equi-struc-trans} that results in {\it Rules J1-J6} from Theorem \ref{struc-equi}.

    {\it Rule J1}: Assume $k - j \neq \min\big\{\{k_1-j,k_2-j,\dots\}\cup \{k-j_1,k-j_2,\dots\}\big\}$, $j_t\le k$ and $j \le k_u $. Then either of the following two situations occur:
    \begin{itemize}
        \item There exists an $l$ such that $k - l = \min\big\{\{k_1-l,k_2-l,\dots\}\cup \{k-l_1,k-l_2,\dots\}\big\}$, $l_t\le k$ and $l \le k_u $. We apply $L_{l} \oplus L_{k} \rightsquigarrow L_{l-1} \oplus L_{k+1}$. Then, applying G1 from Theorem \ref{equi-struc-trans}, $L_{j} \oplus L_{l-1} \rightsquigarrow L_{j-1} \oplus L_{l}$ violates the cover relation.
        \item There exists an $l$ such that $l-j = \min\big\{\{l_1-j,l_2-j,\dots\}\cup \{l-j_1,l-j_2,\dots\}\big\}$, $j_t\le l$ and $j \le l_u $. Consequently, we apply, $L_j \oplus L_l \rightsquigarrow L_{j-1} \oplus L_{l+1}$ followed by G1 from Theorem \ref{equi-struc-trans} on the blocks $L_{l+1}$ and $L_k$, i.e., $L_{l+1} \oplus L_k \rightsquigarrow L_l \oplus L_{k+1}$. This violates the cover relation.
    \end{itemize}
    Conversely, if $\orb_{P_1}^e$ does not cover $\orb_{P_2}^e$, then there exist a matrix pencil $P$´ such that $\orb_{P_1}^e$ covers $\orb_{P{\text{´}}}^e$ and $\orb_{P{\text{´}}}^e$ covers $\orb_{P_2}^e$. Then we use $J1$ twice to achieve $P_2$ from $P_1$ which violates the minimality of the difference $k-j$. Similar argument holds for {\it Rules J2} and {\it Rules J5}.

    {\it Rule J3}: The requirements of $r_{j+1}(P_1) = 1$ and $j+1$ to be the size of the largest $L$-block in $P_1$ comes from the requirement in {\it C2} (Theorem \ref{coin-equi}) that the rightmost column in $\mathcal{R}$ is a single coin. Since the coin is appended as a new rightmost column of some ${\cal J}_{\mu_i}$, $k$ has to be the size of the largest $E$-block in $P_1$. Again, similar argument holds for {\it Rule J4}.  

    {\it Rule J6}: The requirement $k_i$ to be the size of the largest Jordan block in $P_1$ corresponding to each distinct $\mu_i \in \overline{\mathbb C}$ results from the action of removing the coins from all of the longest (=lowest) rows corresponding to all ${\cal J}_{\mu_i}$ in {\it C4} in Theorem \ref{coin-equi}. Also, the requirements $p\ge p_l$, $q \ge q_j$, for all existing $L_{p_l}$ and $L_{q_j}^T$ in $P_1$ results from the requirements that at least all non-zero columns of ${\cal R}$ and ${\cal L}$ are given coins.
\end{proof}
\subsection{Skew-symmetric matrix pencils under congruence}

A $n \times n$ skew-symmetric matrix pencil $A - \lambda B$ is said to be congruent to the pencil $C - \lambda D$ if and only if there is a non-singular matrix $S$ such that $(S^{T}AS,S^{T}BS)=(C,D)$. The orbit of $A - \lambda B$ under the action of the group $GL_n(\mathbb C)$ on the space of skew-symmetric matrix pencils by congruence is defined as follows:
\begin{equation*} \label{congorbit}
\orb_{A - \lambda B}^c = \{S^T (A - \lambda B) S \ | \ S \in GL_n(\mathbb C)\}.
\end{equation*} The next theorem provides a canonical form for skew-symmetric matrix pencil under congruence transformation.

\begin{theorem}{ \cite{Thom91} }\label{lkh}
Each skew-symmetric $n \times n$ matrix pencil $A - \lambda B$ is congruent 
to a direct sum,
determined uniquely up
to permutation of
summands, of pencils of
the form
$$H_{h}(\mu) = 
\begin{bmatrix}0&J_h(\mu)\\
-J_h(\mu)^T &0
\end{bmatrix}
- \lambda 
\begin{bmatrix}0&I_h\\
-I_h &0
\end{bmatrix}
, \text{ } \mu \in\mathbb C,$$
$$K_k= 
\begin{bmatrix}0&I_k\\
-I_k&0
\end{bmatrix}
- \lambda 
\begin{bmatrix}0&J_k(0)\\
-J_k(0)^T &0
\end{bmatrix}, \text{ }
M_m = 
\begin{bmatrix}0&F_m\\
-F_m^T &0
\end{bmatrix}
- \lambda
\begin{bmatrix}0&G_m\\
-G_m^T&0
\end{bmatrix}.$$
\end{theorem}
\hide{Therefore every skew-symmetric pencil $A - \lambda B$ is congruent to one 
in the following direct sum form  
\begin{equation*}
A - \lambda B=
\bigoplus_{j}^{}
\bigoplus_{i}^{}
H_{h_i}(\mu_j)
 \oplus
\bigoplus_{i}^{}K_{k_i}
 \oplus
\bigoplus_{i}^{}
M_{m_i},
\end{equation*}
where the first direct (double) sum corresponds to all distinct eigenvalues $\mu_j$.}

A matrix pencil can be \textit{skew-symmetrized} if its strict equivalence orbit contains a skew-symmetric matrix pencil. The following theorem is a necessary and sufficient condition for a matrix pencil to be skew-symmetrized.

\begin{lemma} { \cite{DmKa14} }\label{canbeskew}
A matrix pencil can be skew-symmetrized if and only if the following conditions hold:
\begin{enumerate}
\item For every $k$ and each distinct $\mu$, the KCF of the pencil contains 
an even number of blocks of the type $E_k(\mu)$;
\item For every $k$, the KCF contains 
an even number of blocks of the type $E_k(\infty)$;
\item For every $k$, the number of blocks $L_k$ is equal to the number of blocks $L_k^T$ in the KCF.
\end{enumerate}
\end{lemma}

Lemma \ref{sumind} relates the normal rank of a skew-symmetric matrix pencil to the KCF.

\begin{lemma} \label{sumind}\cite{DmKa14}
Let $P$ be a matrix pencil, taken in the KCF, i.e., it is a direct sum of the blocks $E_{a_i}(\lambda),E_{b_i}(\infty),L_{c_i},$ and $L^T_{d_i}$. Then the normal rank of $P$ is equal to the sum of the indices, $a_i,b_i,c_i,$ and $d_i$, of all its Kronecker canonical blocks.
\end{lemma} 
\begin{remark}\label{even-normal-rank-SSMP}
Lemmas \ref{canbeskew} and \ref{sumind} imply that normal rank of a skew-symmetric matrix pencil is an even number.
\end{remark}
Lemma \ref{jordaneven} is an observation about integer partition corresponding to Jordan blocks for skew-symmetric matrix pencils.
\begin{lemma} \label{jordaneven}
Let $P$ be a skew-symmetric matrix pencil. For every $\mu$, each $j_i^{\mu}(P)$ in ${\cal J}_{\mu}(P)$ is even, for all $i = 1,2,\ldots$
\end{lemma} 
\begin{proof}
Let the total number of Jordan blocks corresponding to the eigenvalue $\mu$ of size equal to {\it{i}} be $n_i^{\mu}(P)$. 
So, $j_i^{\mu}(P)$ = $n_i^{\mu}(P) + n_{i+1}^{\mu}(P) + \dots + n_t^{\mu}(P) $, where {\it{t}} is the size of the largest Jordan block. By Lemma \ref{canbeskew}, each $n_i^{\mu}(P)$ is even and the result follows.
\end{proof}
\section{Orbit cover relations for 
skew-symmetric matrix pencils}
\label{orbit-cover-relations}

In this section we develop a theory that helps us to formulate a necessary and sufficient condition for $\orb_{P_1}^c$ to cover $\orb_{P_2}^c$, for two skew-symmetric matrix pencils $P_1$ and $P_2$, in terms of both coin moves and structure transitions of canonical blocks. The following lemma is an observation about closure of orbits of skew-symmetric matrix pencils that will be used in the proof of the main result of this paper.

\begin{lemma}\label{SSMP-nei-MP-nei}
Let $P_1$ and $P_2$ be two skew-symmetric matrix pencils such that $\orb_{P_1}^c$ covers $\orb_{P_2}^c$. Then there exist some $p \times q$ matrix pencils $W_i$ such that rows and columns of $P_i$ can be permuted and written as 
$\begin{bmatrix}
 0 & W_i\\
 -W_i^T & 0
\end{bmatrix}, i = 1,2$ and $\orb_{W_1}^e$ covers $\orb_{W_2}^e$.
\end{lemma}
\begin{proof}
Without loss of generality, we assume that $P_1$ and $P_2$ are in the canonical form. By permutations of the rows and corresponding permutations of the columns, the matrix pencils $ P_i, i=1,2$, can be written as  
\begin{equation}\label{permute_P_i}
\widetilde{P_i} = Q_i^TP_iQ_i=
\begin{bmatrix}
0&W_i\\
- W_i^T&0\\
\end{bmatrix},  
\end{equation} 
where $W_i = X_i - \lambda Y_i$ is a $p\times (n-p)$ pencil and $Q_i$ is a permutation matrix for $i=1,2$.
The block $W_1$ in \eqref{permute_P_i} is the direct sum of the top-right corner blocks of $H$-, $K$-, $M$-summands in the skew-symmetric canonical forms of $P_1$, the rest of the blocks form $-W_1^T$. Let, $f = \frac{r_0(P_2) - r_0(P_1)}{2}$. By Lemma 3.7 in \cite{DmKa14}, $f$ is an integer. The block $W_2$ in \eqref{permute_P_i} is the direct sum of the top-right corner blocks of the $H$- and  $K$-summands in the skew-symmetric canonical form of $P_2$ as well as the $f$ largest $L^T$ and ($r_0(P_2)-f$) smallest $L$-blocks of $P_2$, 
all the remaining blocks of $P_2$ forms $-W_2^T$. Note that the $L$- and $L^T$- blocks are parts of the $M$-summands. 

Now assume $\orb_{W_1}^e$ does not cover $\orb_{W_2}^e$ then there is $W'$ (not equal to $W_1$ or $W_2$) such that 
$$
\overline{\orb_{W_1}^e} \supset \orb_{W'}^e 
\quad \text{and} \quad
\overline{\orb_{W'}^e} \supset \orb_{W_2}^e.$$
Define 
$
P'=
\begin{bmatrix}
0&W'\\
- (W')^{T}&0
\end{bmatrix}.$ 
By \cite[Lemma 3.8]{DmKa14}, we have, 
\begin{equation} 
\overline{\orb_{P_1}^c} \supset \orb_{P'}^c 
\quad \text{and} \quad
\overline{\orb_{P'}^c} \supset \orb_{P_2}^c.
\end{equation}
It remains to prove that $P' \neq P_1$ and $P' \neq P_2$. Since $\overline{\orb_{W_1}^e} \supset \orb_{W'}^e$, we obtain $W'$ from $W_1$ by using a sequence of structure transitions (Theorem \ref{equi-struc-trans}). Thus, $P'$ is obtained from $P_1$ by a sequence of structure transitions of canonical blocks under strict equivalence (Theorem \ref{equi-struc-trans}). Thus $P' \neq P_1$. Similarly, $P' \neq P_2$. Hence $\orb_{P_1}^c$ does not cover $\orb_{P_2}^c$ and we arrived at the contradiction.
\end{proof}


Theorem \ref{congruence} is the main result of the manuscript, that states a necessary and sufficient condition for $\orb_{P_1}^c$ to cover $\orb_{P_2}^c$, for two skew-symmetric matrix pencils $P_1$ and $P_2$. 
\begin{theorem} \label{congruence}
Let $P_1$ and $P_2$ be two skew-symmetric matrix pencils. The congruence orbit $\orb_{P_1}^c$ covers the congruence orbit $\orb_{P_2}^c$ if and only if $P_2$ can be obtained from $P_1$ by applying one of the following four structure transitions to the canonical blocks of $P_1$:
\begin{itemize}  
\item{Rule 1.} $M_{j} \oplus M_{k} \rightsquigarrow M_{j-1} \oplus M_{k+1}$, $1\le j\le k$ such that $k - j = \min\big\{\{k_1-j,k_2-j,\dots\}\cup \{k-j_1,k-j_2,\dots\}\big\}$, $j_t\le k$ and $j \le k_u $;
\item{Rule 2.} $H_{j-1}(\mu) \oplus H_{k+1}(\mu) \rightsquigarrow H_{j}(\mu) \oplus H_{k}(\mu)$, $\mu \in \overline{\mathbb C}$, $1\le j \le k$ such that $k - j = \min\big\{\{k_1-j,k_2-j,\dots\}\cup \{k-j_1,k-j_2,\dots\}\big\}$, $j_t\le k$ and $j \le k_u $;
\item{Rule 3.} $M_{j+1} \oplus H_{k}(\mu) \rightsquigarrow M_{j} \oplus H_{k+1}(\mu)$, $j,k=0,1,2,...$ and $\mu \in \overline{\mathbb C}$, such that $k$ and $j+1$ are the sizes of the largest $H$-block and $M$-block in $P_1$, respectively, provided that $r_{j+1}(P_1) = l_{j+1}(P_1) = 1$; 
\item{Rule 4.} $\bigoplus_{i=1}^tH_{k_i}(\mu_i) \rightsquigarrow M_{x} \oplus M_{y}$, $x\ge z$, $y \ge z$, for all $M_{z}$ in $P_1$; $|x-y| \le 1$; $x+y+1 = \sum_{i=1}^d k_i, k_i$ is the size of the largest Jordan block in $P_1$ corresponding to each distinct $\mu_i \in \overline{\mathbb C}.$
\end{itemize}
\end{theorem} 
\begin{proof}
First, observe that, by Lemma \ref{canbeskew}, $\mathcal{L}(P) =\mathcal{R}(P)$, for a skew-symmetric matrix pencil $P$. So we consider only one of them, unless mentioned. As in Lemma \ref{SSMP-nei-MP-nei}, the rows and columns of $P_i$ are permuted such that 
\begin{equation*} \label{permute}
    \widetilde P_i =\begin{bmatrix}
 0 & W_i\\
 -W_i^T & 0
\end{bmatrix}, \space \text{for } i = 1,2.
\end{equation*} 
We start by proving the necessary part. Since $\orb_{P_1}^c$ covers $\orb_{P_2}^c$, $\orb_{W_1}^e$ covers $\orb_{W_2}^e$ (Lemma \ref{SSMP-nei-MP-nei}). Thus, in particular, there exist non-singular matrices $Q$ and $R$ and an arbitrarylly small (entrywise) matrix $E$ such that $
Q^{-1}(W_2 + E)R = W_1, \text{ and, respectively, } R^T(-W_2 - E)^TQ^{-T} = -W_1^T
.$ Combining the two equations above, we get,
\begin{equation}\label{main_eq}
\begin{bmatrix}
   0 & W_1\\
   -W_1^T & 0
\end{bmatrix} =
\begin{bmatrix}
   Q^{-T} & 0\\
   0 & R
\end{bmatrix}^T
\begin{bmatrix}
   0 & W_2 + E\\
   -W_2^T - E^T & 0
\end{bmatrix}
\begin{bmatrix}
   Q^{-T} & 0\\
   0 & R
\end{bmatrix}.
\end{equation}
Let N =
{\small$\begin{bmatrix} 
  Q^{-T} & 0\\
   0 & R
\end{bmatrix}$}, $F =$ 
\small{
$\begin{bmatrix} 
  0 & E\\
  -E^T & 0 
\end{bmatrix}$},
$W_1 = \bigoplus_{j}^{}
\bigoplus_{i}^{}
E_{h_i}(\mu_j)
 \oplus
(\bigoplus_{i}^{}
L_{m_i}), \mu_j \in \overline{\mathbb C}.$ \hide{When $nrk(P_2) = nrk(P_1)$, then $f=0$, i.e., the total number of $L$-blocks/$L^{T}$-blocks remains fixed.
The relation between integer partitions of $P_i$ and $W_i$ are as follows:
\begin{equation}\label{relation}
    {\cal R} (W_i) = {\cal R} (P_i), \text{ } 
{\cal L} (W_i) = 0, \text{ } \mathcal{J}_{\mu_{j}}(W_i) = \dfrac{1}{2} \mathcal{J}_{\mu_{j}}(P_i), \text{ for all } j; \text{ }
i = 1,2.
\end{equation}
When $nrk(P_2) \neq nrk(P_1)$, then $r_0(P_1) \neq r_0(P_2)$ and this implies $f \neq 0$.} Again, since $\orb_{W_1}^e$ covers $\orb_{W_2}^e$, by Theorem \ref{struc-equi}, and taking into account the absence of $L^T$-blocks in $W_1$, one of the following four cases holds for $W_1$ and $W_2$.\\
\textbf{Case 1}: \hide{$nrk(W_2) = nrk(W_1)$, ${\cal R}(W_1) \succ {\cal R} (W_2)$, and ${\cal J}_{\mu_j}(W_1) = {\cal J}_{\mu_j}(W_2)$, for all {\it{j}}.
In this case,} $W_2$ is obtained from $W_1$ by using {\it{J1}} from Theorem \ref{struc-equi}, i.e., 
\begin{equation}\label{expression_1}
    L_{m_p} \oplus L_{m_q} \rightsquigarrow L_{m_p-1} \oplus L_{m_q+1}, \text{ } 1 \le m_p \le m_q,
\end{equation} 
such that $m_q - m_p = \min\big\{\{m_{q_1}-m_p,m_{q_2}-m_p,\dots\}\cup \{m_q-m_{p_1},m_q-m_{p_2},\dots\}\big\}$, $m_{p_t}\le m_q$ and $m_p \le m_{q_u} $. The block $W_1$ can be rewritten as  $$(L_{m_p} \oplus L_{m_q}) \oplus (\bigoplus_{j}^{}
\bigoplus_{i}^{}
E_{h_i}(\mu_j))
 \oplus
(\bigoplus_{k}^{}
L_{m_k}).$$
After applying \eqref{expression_1} to $W_1$, we get the block $W_2$, that can be rewritten as $$(L_{m_p-1} \oplus L_{m_q+1}) \oplus (\bigoplus_{j}^{}
\bigoplus_{i}^{}
E_{h_i}(\mu_j))
 \oplus
(\bigoplus_{k}^{}
L_{m_k}).$$
Let K = $\bigoplus_{j}^{} \bigoplus_{i}^{} E_{h_i}(\mu_j) \oplus (\bigoplus_{k}^{} L_{m_k})$, $k \neq p, q,$ and $X = $
\small{$\begin{bmatrix} 
  0 & K\\
  -K^T   & 0
\end{bmatrix}$}. 
Using the expressions for $W_1$ and $W_2$, we get,
$$ \widetilde P_1 =
\begin{bmatrix}
   0 & W_1\\
   -W_1^T & 0
\end{bmatrix} =
\begin{bmatrix}
   0 & L_{m_p} \oplus L_{m_q}\oplus K\\
   -(L_{m_p} \oplus L_{m_q} \oplus K)^T & 0
\end{bmatrix} \text{ and }
$$
$$ \widetilde P_2 =
\begin{bmatrix}
   0 & W_2\\
   -W_2^T & 0
\end{bmatrix} =
\begin{bmatrix}
   0 & L_{m_p-1} \oplus L_{m_q+1}\oplus K\\
   -(L_{m_p-1} \oplus L_{m_q+1} \oplus K)^T & 0
\end{bmatrix}.$$ 
By \eqref{main_eq} $\widetilde P_{1} = N^{T}(\widetilde P_{2} + F)N. $ The rows and the columns of $\widetilde P_1$ and $\widetilde P_2$ can be permuted by some permutation matrices $Y_1$ and $Y_2$ such that
$Y_1^{T}\widetilde P_{1}Y_1 = Y_1^{T}N^{T}Y_2(Y_2^{T}\widetilde P_{2}Y_2 + Y_2^{T}FY_2)Y_2^{T}NY_1 $ or, equivalently, $$\hat{P_{1}} = (Y_2^{T}NY_1)^{T}(\hat{P_{2}} + F')(Y_2^{T}NY_1), \text{ where } F' = Y_2^{T}FY_2,$$
$$\hat{P_{1}} =
\begin{bmatrix}
   0 & L_{m_p}\\
   -L_{m_p}^T & 0
\end{bmatrix} \oplus
\begin{bmatrix}
   0 & L_{m_q}\\ 
   -L_{m_q}^T & 0
\end{bmatrix}\oplus
\begin{bmatrix}
   0 & K\\
   -K^T & 0
\end{bmatrix} = 
M_{m_p} \oplus M_{m_q} \oplus X, \text{ and }
$$
$$\hat{P_{2}} =
\begin{bmatrix}
   0 & L_{m_p-1}\\
   -L_{m_p-1}^T & 0
\end{bmatrix} \oplus
\begin{bmatrix}
   0 & L_{m_q+1}\\ 
   -L_{m_q+1}^T & 0
\end{bmatrix} \oplus
\begin{bmatrix}
   0 & K\\
   -K^T & 0
\end{bmatrix} = M_{m_p-1} \oplus M_{m_q+1} \oplus X.
$$
Since $\hat{P_{1}} \text{ and } \hat{P_{2}}$ are, respectively, congruent to $P_{1} \text{ and } P_{2}$, the structure transition from $P_1$ to $P_2$ is of the type $$
M_{m_p} \oplus M_{m_q} \rightsquigarrow M_{m_p-1} \oplus M_{m_q+1},$$ such that $m_q - m_p = \min\big\{\{m_{q_1}-m_p,m_{q_2}-m_p,\dots\}\cup \{m_q-m_{p_1},m_q-m_{p_2},\dots\}\big\}$, $m_{p_t}\le m_q$ and $m_p \le m_{q_u} $. Note that, later in the proof, the notations $N$, $F'$, $Y_{i}$, $\hat{P_{i}}, i = 1,2,$ are local to the particular case in which they will be mentioned.\\
\textbf{Case 2}: \hide{$nrk(W_2) = nrk(W_1)$, ${\cal R} (W_2) = {\cal R} (W_1)$, ${\cal J}_{\mu_j}(W_2) \succ {\cal J}_{\mu_j}(W_1)$, for some eigenvalue ${\mu_j}$. Note, ${\cal J}_{\mu_i}(W_2) \succ {\cal J}_{\mu_i}(W_1)$,  and ${\cal J}_{\mu_j}(W_2) \succ {\cal J}_{\mu_j}(W_1)$, for $\mu_i \neq \mu_j$, contradicts $\orb_{W_1}^e$ covers $\orb_{W_2}^e$. Now,}$W_2$ is obtained from $W_1$ by using {\it{J5}} from Theorem \ref{struc-equi}: 
\begin{equation}\label{expression_2}
        E_{h_p-1}(\mu_1) \oplus E_{h_q+1}(\mu_1) \rightsquigarrow E_{h_p}(\mu_1) \oplus E_{h_q}(\mu_1), \space 1 \le h_p \le h_q,
\end{equation}
such that $h_q - h_p = \min\big\{\{h_{q_1}-h_p,h_{q_2}-h_p,\dots\}\cup \{h_q-h_{p_1},h_q-h_{p_2},\dots\}\big\}$, $h_{p_t}\le h_q$ and $h_p \le h_{q_u} $. The block $W_1$ can be rewritten as $$(E_{h_p-1}(\mu_1) \oplus E_{h_q+1}(\mu_1)) \oplus (\bigoplus_{j}^{}
\bigoplus_{i}^{}
E_{h_i}(\mu_j))
\oplus
(\bigoplus_{i}^{}
L_{m_i}), \text{ } j \neq 1.$$ 
After applying \eqref{expression_2} to $W_1$, we obtain the block $W_2$, that can be rewritten as $$(E_{h_p}(\mu_1) \oplus E_{h_q}(\mu_1)) \oplus (\bigoplus_{j}^{}
\bigoplus_{i}^{}
E_{h_i}(\mu_j))
 \oplus
(\bigoplus_{i}^{}
L_{m_i}),  \text{ } j \neq 1.$$
Let K $= \bigoplus_{j}^{}
\bigoplus_{i}^{}
E_{h_i}(\mu_j)
 \oplus
(\bigoplus_{i}^{}
L_{m_i}), (j,i) \neq (1,q)$, $(j,i) \neq (1,p)$, and $X =$  
\small{$\begin{bmatrix} 
  0 & K\\
  -K^T & 0
\end{bmatrix}$}. Similarly to Case 1, $\widetilde P_1$, $\widetilde P_2$, constructed from the respective blocks $W_1$, $W_2$ can be transformed using some permutation matrices $Y_{1}$ and $Y_{2}$, respectively, such that $Y_1^{T}\widetilde P_{1}Y_1 = Y_1^{T}N^{T}Y_2(Y_2^{T}\widetilde P_{2}Y_2 + Y_2^{T}FY_2)Y_2^{T}NY_1 $, or, equivalently,
$$ \hat{ P_{1}} = (Y_2^{T}NY_1)^{T}(\hat{P_{2}} + F')(Y_2^{T}NY_1),\text{ where } F' = Y_2^{T}FY_2,$$ $$\hat{ P_{1}} = H_{h_p-1}(\mu_1) \oplus H_{h_q+1}(\mu_1) \oplus X \text{ and } \hat{P_{2}} = H_{h_p}(\mu_1) \oplus H_{h_q}(\mu_1) \oplus X. $$ Again, similarly to Case 1, the required structure transition from $P_1$ to $P_2$ is of the type $$
H_{h_p-1}(\mu) \oplus H_{h_q+1}(\mu) \rightsquigarrow H_{h_p}(\mu) \oplus H_{h_q}(\mu)$$ where $\mu \in \overline{\mathbb C}$ and $h_q - h_p = \min\big\{\{h_{q_1}-h_p,h_{q_2}-h_p,\dots\}\cup \{h_q-h_{p_1},h_q-h_{p_2},\dots\}\big\}$, $h_{p_t}\le h_q$ and $h_p \le h_{q_u} $.\\
\textbf{Case 3}: \hide{$\nrk(W_2)  = \nrk(W_1)$, ${\cal J}_{\mu_j}(W_2) \succ {\cal J}_{\mu_j}(W_1)$, for some eigenvalue ${\mu_j}$, and ${\cal R}(W_1) \succ {\cal R} (W_2)$. Note, any sequential use of {\it{Rule 1}} and {\it{Rule 2}} will contradict to our assumption that $\orb_{W_1}^c$ covers $\orb_{W_2}^c$. Now} We obtain $W_2$ from $W_1$ by using the third structure transition, i.e., {\it{J3}} from Theorem \ref{struc-equi}:
\begin{equation}\label{expression_3}
L_{m_p+1} \oplus E_{h_q}(\mu_1) \rightsquigarrow L_{m_p} \oplus E_{h_q+1}(\mu_1),\end{equation} where $m_p+1$ and $h_q$ are the sizes of the largest $L$-block and $E$-block, respectively, of $W_1$, such that $r_{m_p+1}(W_1) = 1$. Thus, the block $W_1$ can be rewritten as
$$L_{m_p+1} \oplus E_{h_q}(\mu_1) \oplus (\bigoplus_{j}^{}
\bigoplus_{i}^{}
E_{h_i}(\mu_j))
 \oplus
(\bigoplus_{k}^{}
L_{m_k}),  \text{ } j \neq 1.$$
After applying \eqref{expression_3} to $W_1$, we get the block $W_2$, that can be rewritten as
$$L_{m_p} \oplus E_{h_q+1}(\mu_1) \oplus (\bigoplus_{j}^{}
\bigoplus_{i}^{}
E_{h_i}(\mu_j))
 \oplus
(\bigoplus_{k}^{}
L_{m_k}),  \text{ } j \neq 1.$$
Let K = $\bigoplus_{j}^{} \bigoplus_{i}^{} E_{h_i}(\mu_j) \oplus (\bigoplus_{k}^{} L_{m_k}), (j,i) \neq (1,q)$, $k \neq p \text{ and } X = $ \small{$ 
\begin{bmatrix} 
  0 & K\\
  -K^T & 0
\end{bmatrix}$}. Similarly to Case 1, $\widetilde P_1$, $\widetilde P_2$, constructed from the respective blocks $W_1$, $W_2$ can be transformed using some permutation matrices $Y_{1}$ and $Y_{2}$, respectively, such that 
$Y_1^{T}\widetilde P_{1}Y_1 = Y_1^{T}N^{T}Y_2(Y_2^{T}\widetilde P_{2}Y_2 + Y_2^{T}FY_2)Y_2^{T}NY_1 $, or, equivalently,
$$\hat{ P_{1}} = (Y_2^{T}NY_1)^{T}(\hat{P_{2}} + F')(Y_2^{T}NY_1), \text{ where } F' = Y_2^{T}FY_2,$$ $$\hat{ P_{1}} = M_{m_p+1} \oplus H_{h_q}(\mu_1) \oplus X, \text{ and } \hat{P_{2}} = M_{m_p} \oplus H_{h_q+1}(\mu_1) \oplus X.$$ Again, similarly to Case 1, the required structure transition from $P_1$ to $P_2$ is of the type
$$M_{m_p+1} \oplus H_{h_q}(\mu) \rightsquigarrow M_{m_p} \oplus H_{h_q+1}(\mu),$$ where $m_p,h_q=0,1,2,...$ and $\mu \in \overline{\mathbb C}$, such that $h_q$ and $m_p+1$ are the sizes of largest $H$-block and $M$-block of $P_1$ respectively, provided that $r_{m_p+1}(P_1) = l_{m_p+1}(P_1) = 1$.

\textbf{Case 4}: \hide{nrk($W_2$) $\neq$ nrk($W_1$), i.e., normal rank changes.}$W_2$ is obtained from $W_1$ using {\it{J6}} from Theorem \ref{struc-equi}:
\begin{equation}\label{expression_4}
    \bigoplus_{k=1}^tE_{h_k}(\mu_k) \rightsquigarrow L_{m_p} \oplus L_{m_q}^T, \mu_i \neq \mu_j
\end{equation} $m_p \ge m_l$, $m_q \ge m_k$, for all existing $L_{m_{l}}$ and $L^{T}_{m_{k}}$ in $P_1$; $t$ is total number of eigenvalues and each $h_{k}$ is the size of the largest $E$-block in $P_1$ corresponding to $\mu_k \in \overline{\mathbb C}$; $\sum_{k=1}^t h_{k} = m_p + m_q + 1$. The block $W_1$ can be rewritten as  $$(\bigoplus_{k=1}^tE_{h_k}(\mu_k)) \oplus (\bigoplus_{j}^{}
\bigoplus_{i}^{}
E_{h_i}(\mu_j))
 \oplus
(\bigoplus_{l}^{}
L_{m_l}).$$
After applying \eqref{expression_4} to $W_1$, we obtain the block $W_2$, that can be rewritten as 
$$(L_{m_p} \oplus L_{m_q}^T) \oplus (\bigoplus_{j}^{}
\bigoplus_{i}^{}
E_{h_i}(\mu_j))
 \oplus
(\bigoplus_{l}^{}
L_{m_l}),  \text{ } i \neq k.$$
Let K = $\bigoplus_{j}^{} \bigoplus_{i}^{} E_{h_i}(\mu_j) \oplus (\bigoplus_{l}^{} L_{m_l})$, $i \neq k$ and X = 
\small{$\begin{bmatrix} 
  0 & K\\
  -K^T & 0
\end{bmatrix}$}. As in Case~1, $\widetilde P_1$, $\widetilde P_2$, constructed from the respective blocks $W_1$, $W_2$ can be transformed using some permutation matrices $Y_{1}$ and $Y_{2}$, respectively, such that $Y_1^{T}\widetilde P_{1}Y_1 = Y_1^{T}N^{T}Y_2(Y_2^{T}\widetilde P_{2}Y_2 + Y_2^{T}FY_2)Y_2^{T}NY_1 $, or, equivalently,
$\hat{ P_{1}} = (Y_2^{T}NY_1)^{T}(\hat{P_{2}} + F')(Y_2^{T}NY_1),\text{ where }$ $$F' = Y_2^{T}FY_2, \text{ }\hat{ P_{1}} = \bigoplus_{k=1}^t H_{h_k}(\mu_k) \oplus X \text{ and } \hat{P_{2}} = M_{m_p} \oplus M_{m_q} \oplus X. $$ We now show that $|m_p-m_q| \le 1$. Assume, $|m_p-m_q| > 1$, then we apply {\it{Rule 1}} as many times as required and get $M_{m_p} \oplus M_{m_q} \oplus X$, where $|m_p-m_q| \le 1$. Hence, the structure transition from $P_1$ to $P_2$ is
$$
\bigoplus_{k=1}^tH_{h_k}(\mu_k) \rightsquigarrow M_{m_p} \oplus M_{m_q},$$ where $m_p \ge m_r$, $m_q \ge m_r$, for all existing $M_{m_r}$ in $P_1$; $m_p+m_q+1 = \sum_{k=1}^t h_k, \text{ } h_k$ is the size of the largest $H$-block in $P_1$ corresponding to each distinct $\mu_k \in \overline{\mathbb C};$ $|m_p-m_q| \le 1$.

For proving the converse of the theorem, we first prove that no rule can replace any finite sequence of {\it{Rules 1-4}}. Assume $\{a,b,c,d\} = \{1,2,3,4\}$ and there exist a finite sequence of rules that gives $P_2$ from $P_1$ using, for e.g., {\it{Rule b}},  $$Rule\text { } a \xrightarrow{} ... \xrightarrow{} Rule\text { } b \xrightarrow{} ... \xrightarrow{}  Rule\text { } c \xrightarrow{} ... \xrightarrow{} Rule\text { } d \iff Rule\text { } b.$$ Let {\it{Rule b}} be the $k^{th}$ element of the sequence. We can replace {\it{Rule b}} at the $k^{th}$ position with the same sequence any number of times which implies existence of an arbitrary long downward path. The latter is not possible because stratification graph has finite number of vertices and edges. So, any sequence of rules that replaces {\it{Rule i}}, $i \in \{1,2,3,4\}$, does not contain {\it{Rule i}}. Now, while navigating from $P_1$ to $P_2$ in the graph, either normal rank changes or it remains same. Since {\it{Rules 1-3}} do not change the normal rank, {\it{Rule 4}} can not be a part of any sequence that replaces any of {\it{Rules 1-3}}. We now explain {\it{Rule i}} cannot replace any sequence of other rules (not containing {\it{Rule i}}), separately for every rule.\\
{\it{Rule 1}}: Assume {\it{Rule 1}} can be replaced by a sequence of {\it{Rules 2-3}}. This is clearly not possible because use of {\it{Rule 3}} decreases the size of the only existing largest $L$- and $L^T$- blocks in $P_1$ which does not lead us to the integer partition $\mathcal{R}(P_2)$ that we get after using {\it{Rule 1}}.\\
{\it{Rule 2}}: This case can be argued similar to that of {\it{Rule 1}} but considering $\mathcal{J}_\mu$, for some $\mu \in \overline{\mathbb{C}}$, instead of $\mathcal{R}$.\\
{\it{Rule 3}}: Assume {\it{Rule 3}} can be replaced by a sequence of {\it{Rules 1-2}}. As mentioned in the case of {\it{Rule 1}}, {\it{Rule 3}} decreases the size of the only existing largest $L$- and $L^T$-blocks in $P_1$ but {\it{Rule 1}} will always increase the size of the larger $L$- and $L^T$-block at every step.\\
{\it{Rule 4}}: No sequence of {\it{Rules 1-3}} can replace {\it{Rule 4}} because {\it{Rule 4}} changes normal rank.

Finally, we prove the converse of our theorem. Assume there exist a skew-symmetric matrix pencil $P{\text{´}}$, $P_1 \neq P{\text{´}}$ and $P{\text{´}}\neq P_2$ such that $\orb_{P{\text{´}}}^{c}$ covers $\orb_{P_2}^{c}$, and is covered by $\orb_{P_1}^{c}$. By forward implication, $P{\text{´}}$ can be obtained from $P_1$, and $P_2$ can be obtained from $P{\text{´}}$ using one of {\it{Rules 1-4}} contradicting that $P_2$ can be obtained from $P_1$ using one of {\it{Rules 1-4}}.
\end{proof}
\begin{remark}\label{diff-rowindices}
Let $P_1$ and $P_2$ be two skew-symmetric matrix pencils such that $\orb_{P_1}^c$ covers $\orb_{P_2}^c$, then either $r_0(P_2) - r_0(P_1)$ is equal to $2$ or to $0$. The converse of the statement is not true because a sequence of transitions of the first three types from Theorem \ref{congruence} can result in $r_0(P_2) - r_0(P_1) = 0$ but $\orb_{P_1}^c$ does not cover $\orb_{P_2}^c$ and a sequence of structure transition of all the types from Theorem~ \ref{congruence} in which transition 4 is taken once will result in $r_0(P_2) - r_0(P_1) = 2$, but again $\orb_{P_1}^c$ does not cover $\orb_{P_2}^c$.
\end{remark}

Dropping the restrictions on the indices of the canonical blocks in Theorem \ref{congruence}, we get the following result, analogous to Theorem \ref{equi-struc-trans}, but for skew-symmetric matrix pencils under congruence instead of strict equivalence. 

\begin{theorem}\label{SSMP_struc_trans}
Let $P_1$ and $P_2$ be two skew-symmetric matrix pencils. Then $\orb_{P_2}^c \subset \overline{\orb_{P_1}^c}$ if and only if $P_2$ can be obtained from $P_1$ after changing the canonical blocks of $P_1$ by applying a sequence of structure transitions and each transition is one of the four types below: 
\begin{itemize}  
\item $M_{j} \oplus M_{k} \rightsquigarrow M_{j-1} \oplus M_{k+1}$, $1\le j \le k;$ 
\item $H_{j-1}(\mu) \oplus H_{k+1}(\mu) \rightsquigarrow H_{j}(\mu) \oplus H_{k}(\mu)$, $1\le j \le k$ and $\mu \in \overline{\mathbb C};$
\item $M_{j+1} \oplus H_{k}(\mu) \rightsquigarrow M_{j} \oplus H_{k+1}(\mu)$, $j, \space k=0,1,2, \dots$ and $\mu \in \overline{\mathbb C};$ 
\item $\bigoplus_{i=1}^tH_{k_i}(\mu_i) \rightsquigarrow M_{p} \oplus M_{q}$, if $p+q+1= \sum_{i=1}^t k_i$ and $\mu_i \neq \mu_{i'}$ for $i \neq i', \mu_i \in \overline{\mathbb C}.$ 
\end{itemize}
\end{theorem}

Define a \textit{vertical pair of coins} to be a pair of coins positioned in the same column of the table and in the rows $1+2k$ and $2+2k$, for $k=0,1,\dots$ Figure~\ref{fig:Vertical} illustrates the meaning. 
\begin{figure}[h]
\begin{center}
    \begin{tikzpicture} 
    \draw[color=blue] (0,0) circle (0.1cm);  
    \draw[color=blue] (0.35,0) circle (0.1cm);  
    \draw[color=blue] (0.7,0) circle (0.1cm);
    \draw[color=blue] (1.05,0) circle (0.1cm);
    \draw[color=blue] (0.95,0) -- (1.15,0);
    \draw[color=blue] (1.05,-0.1) -- (1.05,0.1);
    \draw[color=white] (1.3,0) -- (1.5,0);
    \draw[color=white] (1.4,-0.1) -- (1.4,0.1);
    \end{tikzpicture}

    \begin{tikzpicture} 
    \draw[color=blue] (0,0) circle (0.1cm);  
    \draw[color=blue] (0.35,0) circle (0.1cm);  
    \draw[color=blue] (0.7,0) circle (0.1cm);
    \draw[color=blue] (1.05,0) circle (0.1cm);
    \draw[color=blue] (0.95,0) -- (1.15,0);
    \draw[color=blue] (1.05,-0.1) -- (1.05,0.1);
    \draw[color=white] (1.3,0) -- (1.5,0);
    \draw[color=white] (1.4,-0.1) -- (1.4,0.1);
    \end{tikzpicture}
    
    \begin{tikzpicture} 
    \draw[color=blue] (0,0) circle (0.1cm);  
    \draw[color=blue] (0.35,0) circle (0.1cm);  
    \draw[color=blue] (0.7,0) circle (0.1cm);
    \draw[color=blue] (1.05,0) circle (0.1cm);
    \draw[color=blue] (1.4,0) circle (0.1cm);
    \end{tikzpicture}

    \begin{tikzpicture} 
    \draw[color=blue] (0,0) circle (0.1cm);  
    \draw[color=blue] (0.35,0) circle (0.1cm);  
    \draw[color=blue] (0.7,0) circle (0.1cm);
    \draw[color=blue] (1.05,0) circle (0.1cm);
    \draw[color=blue] (1.4,0) circle (0.1cm);
    \end{tikzpicture}
\end{center}
\caption{Integer partition, $\mathcal{J} = \{4,4,4,4,2\}$, 
with vertical pair of coins marked as \coin.
}
\label{fig:Vertical}
\end{figure}
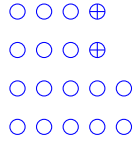


We present Theorem \ref{coin-congruence} which is an interpretation of Theorem \ref{congruence} in terms of coin moves.

\begin{theorem}\label{coin-congruence}
Let $P_1$ and $P_2$ be two skew-symmetric matrix pencils. Then $\orb_{P_1}^c$ covers $\orb_{P_2}^c$ if and only if the integer partitions, i.e., ${\cal R}$, ${\cal L}$ and ${\cal J}_{\mu_{i}}$, corresponding to $P_2$ can be obtained by applying one of the four rules to the integer partitions of $P_1$:
\begin{itemize}
    \item{Rule I:} Minimum rightward coin move from column $j$ in ${\cal R}$ and the same coin move in ${\cal L}$, where $j\ge 1$;
    \item{Rule II:} Minimum leftward coin move of a vertical pair of coins in any ${\cal J}_{\mu_i}$;
    \item{Rule III:} If the rightmost column in each of ${\cal R}$ and ${\cal L}$ is a single coin, append these two coins (together) as a new rightmost column of some ${\cal J}_{\mu_i}$;
    \item{Rule IV:} Let $2k = (\sum_{i=1}^t 2k_i$), where $k_i$ denotes the total number of coins in each of the two longest ($=$lowest) row from each ${\cal J}_{\mu_i}$. Remove these $2k$ coins add two more coins to the set. The following two step distribution is done:
    \begin{itemize}
        \item{} Distribute $k+1$ coins from the set of $2k+2$ coins to $r_p$, $p = 0,\dots,t$ and $l_q$, $q = 0,\dots,k-t-1$,
        \item{} Distribute the remaining $k+1$ coins from the set to $r_p$, $p = 0,\dots,k-t-1$ and $l_q$, $q = 0,\dots,t$,
    \end{itemize} such that at least the existing columns of each of ${\cal R}(P_1)$ and ${\cal L}(P_1)$ receive $2$ coins and $|p-q|$ is the minimum of all possible differences.
\end{itemize}
\end{theorem}
\begin{proof} Assume $\orb_{P_1}^c$ covers $\orb_{P_2}^c$. {\it{Rules 1-4}} from Theorem \ref{congruence} are the four possible rules for structure transitions in the canonical blocks of $P_1$ that allow us to obtain $P_2$. We prove that {\it{Rules 1-4}}, from Theorem \ref{congruence}, imply {\it{Rules I-IV}} respectively.

By {\it{Rule 1}}, Theorem \ref{congruence}, $P_2$ can be obtained from $P_1$ by applying $M_{j} \oplus M_{k} \rightsquigarrow M_{j-1} \oplus M_{k+1}$, $1\le j\le k$ such that $k - j = \min\big\{\{k_1-j,k_2-j,\dots\}\cup \{k-j_1,k-j_2,\dots\}\big\}$, $j_t\le k$ and $j \le k_u $. Let, $\mathcal{R}(P_1) = (r_0,\dots,r_{j-1},r_{j},r_{j+1},\dots,r_{k},r_{k+1},r_{k+2},...)$. The following are the changes in the integer partitions $\mathcal{R}$, $\mathcal{L}$, and $\mathcal{J}$, after the structure transition: \\
We start with $\mathcal{R}$. Note that by definition of Weyr characteristics, $r_{j-1}$ is the total number of $L$-blocks with indices greater than or equal to $j-1$. So the replacement of block $M_j$ by $M_{j-1}$ does not change the value of $r_{j-1}$, i.e., $r_{j-1}(P_1) = r_{j-1}(P_2)$, but $r_j$ decreases by $1$, i.e., $r_{j}(P_2) = r_{j}(P_1)-1$. Similar arguments can be used to explain $r_{k}(P_1) = r_{k}(P_2)$ but $r_{k+1}(P_2) = r_{k+1}(P_1)+1$. All other {\it{parts}} of the integer partition $\mathcal{R}$ remain the same. Thus, $\mathcal{R}(P_2) = (r_0,\dots,r_{j-1},r_{j}-1,r_{j+1},\dots,r_{k},r_{k+1}+1,r_{k+2},...)$. 

Now, we show that $\mathcal{R}(P_2)$ is an integer partition. We consider two cases, namely, $k>j$ and $k=j$. In the first case, the the condition $k - j = \min\big\{\{k_1-j,k_2-j,\dots\}\cup \{k-j_1,k-j_2,\dots\}\big\}$, $j_t< k$ and $j < k_u $ implies 
\begin{equation}\label{inequality_of_indices}
\dots \ge r_{j-1} \ge r_{j} > r_{j+1}=...=r_{k-1} = r_{k} > r_{k+1} \ge r_{k+2} \ge \dots
\end{equation} Note, if $r_{k} = r_{k+1}$ , then there exists no $M_k$ block.\hide{Similar argument implies $r_{j} > r_{j+1}$ and $r_{j+1}=...=r_{k-1} = r_{k}$.} After applying {\it Rule 1} to \eqref{inequality_of_indices}, we get,
$$ \dots \ge r_{j-1} > r_{j} -1 \ge r_{j+1}=...=r_{k-1} = r_{k} \ge r_{k+1} +1 > r_{k+2}\ge \dots$$
Now, if $k=j$ then at least two $M_k$-blocks exist, i.e., 
\begin{equation}\label{equality_of_indices}
\dots \ge r_{j-1} \ge r_{j} = r_{k} > r_{k+1} \ge r_{k+2} \ge \dots, \text{ where } r_k-r_{k+1} \ge 2.
\end{equation} Again applying {\it Rule 1} to \eqref{equality_of_indices}, we get,
$$ \dots \ge r_{j-1} > r_{j} -1 = r_{k}-1 \ge r_{k+1} +1 > r_{k+2}\ge \dots$$
In both the scenerios, $\mathcal{R}(P_2)$ remains to be an integer partition. $\mathcal{L}(P_2)$ has the same partition as that of $\mathcal{R}(P_2)$ and $\mathcal{J}_\mu(P_2) = \mathcal{J}_\mu(P_1)$, for all $\mu \in \overline{\mathbb{C}}$.  Note, by Lemma \ref{canbeskew}, the changes in the integer partitions do not violate the assumed structure of $P_2$.

The above changes in the integer partitions is equivalent to a rightward  coin move in $\mathcal{R}$ from column $j$ to $k+1$ and the same move in $\mathcal{L}$ such that $j \ge 1$. The condition $k - j = \min\big\{\{k_1-j,k_2-j,\dots\}\cup \{k-j_1,k-j_2,\dots\}\big\}$, $j_t\le k$ and $j \le k_u $, makes the rightward coin move a minimum rightward coin move.

Forward implication of {\it{Rules 2}} and {\it{3}} can be stated, respectively, in terms of coin move as {\it{Rules II}} and {\it{III}} using arguments similar to that of {\it{Rule~1}}.
The application of {\it{Rule 2}} to the canonical blocks of $P_1$ yields the following sequences of integers associated with $P_2$: 
\begin{itemize}
    \item $\mathcal{J}_\mu(P_2) = (j_1,\dots,j_{m-1},j_{m}+2,j_{m+1},\dots,j_{k},j_{k+1}-2,j_{k+2}\dots)$, for some $\mu \in \overline{\mathbb{C}}$, where $j_l = j_l(P_1)$.
    \item $\mathcal{R}(P_2) = \mathcal{R}(P_1) = \mathcal{L}(P_1) = \mathcal{L}(P_2)$.
\end{itemize}
Again, application of {\it{Rule 3}} to the canonical blocks of $P_1$ results in the following sequences of integers associated with $P_2$:
\begin{itemize}
    \item $\mathcal{J}_\mu(P_2) = (j_1(P_1),\dots,j_{k}(P_1),2)$, for some $\mu \in \overline{\mathbb{C}}$
    \item $\mathcal{R}(P_2) = (r_0(P_1),\dots,r_{j}(P_1),0) = \mathcal{L}(P_2)$. Note, $r_{j+1}(P_1) = 1$ which becomes zero after the structure transition
\end{itemize} The resulting sequences of integers after the application of both {\it{Rules 2} and \it{3}} remain integer partitions. This can be understood using similar analysis as that for {\it{Rule 1}}.

We prove the forward implication of {\it{Rule 4}}, i.e., $P_2$ can be obtained from $P_1$ by applying $\bigoplus_{i=1}^dH_{k_i}(\mu_i) \rightsquigarrow M_{p} \oplus M_{q}$, $p\ge z$, $q \ge z$, for all existing $M_{z}$ in $P_1$; $p+q+1 = \sum_{i=1}^d k_i, k_i$ is the size of the largest Jordan block corresponding to each distinct $\mu_i \in \overline{\mathbb C}$ and $|p-q| \le 1$.\hide{ i.e.,
$$\bigoplus_{i=1}^t
\begin{bmatrix}
    0 & E_{k_i}(\mu_i)\\
    -E_{k_i}(\mu_i)^T & 0
\end{bmatrix} \rightsquigarrow
\begin{bmatrix}
    0 & L_{p}\\
    -L_{p}^T & 0
\end{bmatrix} \oplus
\begin{bmatrix}
    0 & L_{q}^T\\
    -L_{q} & 0
\end{bmatrix}.$$} Without loss of generality, let $p\ge q$. {\it{Rule 4}} causes the following changes in the integer partitions $\mathcal{R}(P_1)$, $\mathcal{L}(P_1)$ and, for each $\mu_i$, $\mathcal{J}_{\mu_i}(P_1)$:
\begin{itemize}
 \item The formation of $2$ new $M$-blocks of indices greater than or equal to the indices of all $M$-blocks of $P_1$. In such a situation, either of the following cases occur, for all $M_z(P_1)$,
\begin{itemize}
    \item $p>q\ge z$, i.e., $\mathcal{R}(P_2)$ = $\mathcal{L}(P_2) = (r_0 + 2,\dots,r_{q} + 2, r_{p} + 1),$
    \item $p= q\ge z$, i.e., $\mathcal{R}(P_2)$ = $\mathcal{L}(P_2) = (r_0 + 2,\dots,r_{p} + 2).$
\end{itemize} 

\item The number of Jordan blocks of size greater or equal to $x \le k_i$ decreases by $2$, i.e., for each $\mu_i$, $\mathcal{J}_{\mu_i}(P_2) = (j_1(P_1)-2,\dots,j_{k_i-1}(P_1)-2, j_{k_i}(P_1)-2),$ since $k_i$ is the size of the largest Jordan block corresponding to each distinct $\mu_i \in \overline{\mathbb C}$. 
\end{itemize} 

It is clear, for each $\mu_i$, $\mathcal{J}_{\mu_i}(P_2)$, $\mathcal{R}(P_2)$ and $\mathcal{L}(P_2)$ are integer partitions and the skew-symmetry of $P_2$ is preserved.
    
Thus, we get the above integer partitions corresponding to $P_2$ from integer partitions corresponding to $P_1$ by removal of two longest($=$lowest) row of coins from the integer partition corresponding to Jordan blocks  of $P_1$ for all eigenvalues. Assume a total of $2k$ coins have been removed, add $2$ more to this set of $2k$ coins. Firstly, distribute $k+1$ coins to $r_p$, $p = 0,\dots,t$ and $l_q$, $q = 0,\dots,k-t-1$ such that all (existing) columns of $\mathcal{R}(P_1)$ and $\mathcal{L}(P_1)$ get at least $1$ coin. Finally, distribute the remaining $k+1$ coins to $r_p$, $p = 0,\dots,k-t-1$ and $l_q$, $q = 0,\dots,t$.

The backward implication of the theorem is similar to that of Theorem~\ref{congruence}, except the fact that, here, we consider {\it Rules I-IV} instead {\it Rules 1-4}.
\end{proof}
An illustration of each coin move is given in Example \ref{orbit_ex}. From Theorem \ref{congruence}, we derive the following theorem that presents rules for upward navigation in the orbit stratification graph. Note that Theorem \ref{congruence} had rules for downward navigation in the graph.
\begin{theorem}\label{co-for_go_up}
Let $P_1$ and $P_2$ be two skew-symmetric matrix pencils. Then $\orb_{P_1}^c$ covers $\orb_{P_2}^c$ if and only if $P_1$ can be obtained from $P_2$ by applying one of four types of structure transitions:
\begin{itemize}
    \item{Type 1:} $M_{j-1} \oplus M_{k+1} \rightsquigarrow M_{j} \oplus M_{k}$, $1\le j \le k$, such that $k - j = \min\big\{\{k_1-j,k_2-j,\dots\}\cup \{k-j_1,k-j_2,\dots\}\big\}$, $j_t\le k$ and $j \le k_u $;
    \item{Type 2:} $H_{j}(\mu) \oplus H_{k}(\mu) \rightsquigarrow H_{j-1}(\mu) \oplus H_{k+1}(\mu)$, $1\le j \le k$ and $\mu \in \overline{\mathbb C}$, such that $k - j = \min\big\{\{k_1-j,k_2-j,\dots\}\cup \{k-j_1,k-j_2,\dots\}\big\}$, $j_t\le k$ and $j \le k_u $;
    \item{Type 3:} $M_{j} \oplus H_{k+1}(\mu) \rightsquigarrow M_{j+1} \oplus H_{k}(\mu)$, $j,k=0,1,2,...$ and $\mu \in \overline{\mathbb C}$, such that $k+1$ and $j$ are the sizes of the largest $H$-block and $M$-block respectively, provided that $j_{k+1}^\mu(P_1) = 2$;
    \item{Type 4:} $M_{p} \oplus M_{q} \rightsquigarrow \bigoplus_{i=1}^dH_{k_i}(\mu_i)$, if $p+q+1 = \sum_{i=1}^d k_i,$ $p$ and $q$ are the sizes of the largest $M$-blocks in $P_2$ and $k_i$ is the the size of the largest Jordan blocks corresponding to each distinct $\mu_i \in \overline{\mathbb C}$ in $P_1$. 
\end{itemize}  
\end{theorem}

Following \cite{Dmyt16, Dmyt17, DmKa14, DmKS13}, define dimension of $\orb_{A - \lambda B}^c$ as the dimension of the tangent space to $\orb_{A - \lambda B}^c$ at $A - \lambda B$. The {\it codimension} of $\orb_{A - \lambda B}^c$ is $n^2-n$ minus the dimension of the $\orb_{A - \lambda B}^c$. From now onwards, codimension will be denoted as {\it{cod}}. Below, we give two examples demonstrating the main result of the paper.
\begin{example}\label{orbit_ex}  In this example, all the orbits are considered under congruence transformation. In Figure \ref{fig:Orbit_Ex}, we present the orbit stratification graph of $6\times6$ skew-symmetric matrix pencils, where at each edge of the graph we write the rule it was obtained from (one of {\it Rules I-IV} in Theorem \ref{coin-congruence}). We choose $4$ distinct rules from the figure, written in bold font, and explain the coin moves corresponding to each of them. The bullets (\begin{tikzpicture} 
        \filldraw[color=blue] (0,0) circle (0.1cm); 
        \end{tikzpicture}) represent the coins that are involved in the coin moves and the empty circles  (\begin{tikzpicture} 
        \draw[color=blue] (0,0) circle (0.1cm); 
        \end{tikzpicture}) represent the coins that are not involved in the coin moves.
        
\noindent Coin move corresponding to {\it{Rule I}}: $2M_{1} \rightsquigarrow M_{0} \oplus M_{2}$.
\begin{center}
    \begin{tabular}{|c|c|}
    \hline
    Both $\mathcal{R} \text{ and } \mathcal{L}$ before {\it{Rule I}}: & Both $\mathcal{R} \text{ and }\mathcal{L}$ after {\it{Rule I}}: \\
    \hline
    \begin{tikzpicture} 
    \draw[color=blue] (0,0) circle (0.1cm); 
    \filldraw[color=blue] (0.35,0) circle (0.1cm);
    \end{tikzpicture}    & 
    \begin{tikzpicture} 
    \draw[color=blue] (-0.35,0) circle (0.1cm);
    \draw[color=white] (0,0) circle (0.1cm); 
    \draw[color=white] (0.35,0) circle (0.1cm);
    \end{tikzpicture}\\
    \begin{tikzpicture} 
    \draw[color=blue] (0,0) circle (0.1cm);  
    \draw[color=blue] (0.35,0) circle (0.1cm); 
    \end{tikzpicture} & 
    \begin{tikzpicture} 
    \draw[color=blue] (-0.35,0) circle (0.1cm); 
    \draw[color=blue] (0,0) circle (0.1cm); 
    \filldraw[color=blue] (0.35,0) circle (0.1cm); 
    \end{tikzpicture}\\
    \hline
    $J^{\mu_i}$ before {\it{Rule I}}, for any $\mu_i$: & $J^{\mu_i}$ after {\it{Rule I}}, for any $\mu_i$:\\
    \hline
     - & - \\
     \hline
    \end{tabular}
\end{center}
Coin move corresponding to {\it{Rule II}}: $H_3(\mu_1) \rightsquigarrow H_1(\mu_1) \oplus H_2(\mu_1)$.
    \begin{center}
        \begin{tabular}{|c|c|}
        \hline
         Both $\mathcal{R} \text{ and }\mathcal{L}$ before {\it{Rule II}}: &  Both $\mathcal{R} \text{ and }\mathcal{L}$ after {\it{Rule II}}:\\
         \hline
           -  & - \\
        \hline
        $J^{\mu_{1}}$ before {\it{Rule II}}: & $J^{\mu_{1}}$ after {\it{Rule II}}:\\
        \hline
         & 
        \begin{tikzpicture} 
        \filldraw[color=blue] (0,0) circle (0.1cm); 
        \draw[color=white] (0.35,0) circle (0.1cm);
        \end{tikzpicture}\\
         & 
        \begin{tikzpicture} 
        \filldraw[color=blue] (0,0) circle (0.1cm); 
        \draw[color=white] (0.35,0) circle (0.1cm); 
        \end{tikzpicture}\\
        \begin{tikzpicture} 
        \draw[color=blue] (0,0) circle (0.1cm);  
        \draw[color=blue] (-0.35,0) circle (0.1cm);  
        \filldraw[color=blue] (0.35,0) circle (0.1cm); 
        \end{tikzpicture} & \begin{tikzpicture} 
        \draw[color=blue] (0,0) circle (0.1cm); 
        \end{tikzpicture}
        \begin{tikzpicture} 
        \draw[color=blue] (0,0) circle (0.1cm); 
        \end{tikzpicture}\\
        \begin{tikzpicture} 
        \draw[color=blue] (0,0) circle (0.1cm);  
        \draw[color=blue] (-0.35,0) circle (0.1cm);  
        \filldraw[color=blue] (0.35,0) circle (0.1cm); 
        \end{tikzpicture} & \begin{tikzpicture} 
        \draw[color=blue] (0,0) circle (0.1cm); 
        \end{tikzpicture}
        \begin{tikzpicture} 
        \draw[color=blue] (0,0) circle (0.1cm); 
        \end{tikzpicture}\\
        \hline
        \end{tabular}
    \end{center}
Coin move corresponding to {\it{Rule III}}: $M_0 \oplus M_1 \oplus H_1(\mu_1) \rightsquigarrow 2M_0 \oplus H_2(\mu_1)$.
    \begin{center}
        \begin{tabular}{|c|c|}
        \hline
        Both $\mathcal{R} \text{ and }\mathcal{L}$ before {\it{Rule III}}: &  Both $\mathcal{R} \text{ and }\mathcal{L}$ after {\it{Rule III}}:\\
        \hline
        \begin{tikzpicture} 
        \draw[color=blue] (0,0) circle (0.1cm); 
        \end{tikzpicture}
        \begin{tikzpicture} 
        \filldraw[color=white] (0,0) circle (0.1cm); 
        \end{tikzpicture}  & 
        \begin{tikzpicture} 
        \draw[color=blue] (0,0) circle (0.1cm); 
        \end{tikzpicture}\\
        \begin{tikzpicture} 
        \draw[color=blue] (0,0) circle (0.1cm); 
        \end{tikzpicture}
        \begin{tikzpicture} 
        \filldraw[color=blue] (0,0) circle (0.1cm); 
        \end{tikzpicture}  & 
        \begin{tikzpicture} 
        \draw[color=blue] (0,0) circle (0.1cm); 
        \end{tikzpicture}\\
        \hline
        $J^{\mu_{1}}$ before {\it{Rule III}}: & $J^{\mu_{1}}$ after {\it{Rule III}}:\\
        \hline
        \begin{tikzpicture} 
        \draw[color=blue] (0,0) circle (0.1cm); 
        \end{tikzpicture} & 
        \begin{tikzpicture} 
        \draw[color=blue] (0,0) circle (0.1cm); 
        \end{tikzpicture} 
        \begin{tikzpicture} 
        \filldraw[color=blue] (0,0) circle (0.1cm); 
        \end{tikzpicture}\\
        \begin{tikzpicture} 
        \draw[color=blue] (0,0) circle (0.1cm); 
        \end{tikzpicture} & 
        \begin{tikzpicture} 
        \draw[color=blue] (0,0) circle (0.1cm); 
        \end{tikzpicture} 
        \begin{tikzpicture}
        \filldraw[color=blue] (0,0) circle (0.1cm); 
        \end{tikzpicture}\\
        \hline
        \end{tabular}
    \end{center}
Coin move corresponding to {\it{Rule IV}}: $2M_0 \oplus 2H_1(\mu_1) \rightsquigarrow 4M_0 \oplus H_1(\mu_1)$.
    \begin{center}
        \begin{tabular}{|c|c|c|}
           \hline
           Both $\mathcal{R} \text{ and }\mathcal{L}$ before {\it Rule IV} & Both $\mathcal{R} \text{ and }\mathcal{L}$ after {\it Rule IV}\\
           \hline
            &
           \begin{tikzpicture} 
           \filldraw[color=blue] (0,0) circle (0.1cm); 
           \end{tikzpicture}\\
         &
           \begin{tikzpicture} 
           \filldraw[color=blue] (0,0) circle (0.1cm); 
           \end{tikzpicture}\\
           \begin{tikzpicture} 
           \draw[color=blue] (0,0) circle (0.1cm); 
           \end{tikzpicture} &
           \begin{tikzpicture} 
           \draw[color=blue] (0,0) circle (0.1cm); 
           \end{tikzpicture}\\
           \begin{tikzpicture} 
           \draw[color=blue] (0,0) circle (0.1cm); 
           \end{tikzpicture} & 
           \begin{tikzpicture} 
           \draw[color=blue] (0,0) circle (0.1cm); 
           \end{tikzpicture}\\
           \hline
           $J^{\mu_{1}}$ before {\it Rule IV} & $J^{\mu_{1}}$ after {\it Rule IV}\\
           \hline
           \begin{tikzpicture} 
           \draw[color=blue] (0,0) circle (0.1cm); 
           \end{tikzpicture}  &
           \begin{tikzpicture} 
           \draw[color=white] (0,0) circle (0.1cm); 
           \end{tikzpicture}\\
           \begin{tikzpicture} 
           \draw[color=blue] (0,0) circle (0.1cm); 
           \end{tikzpicture}  &
           \begin{tikzpicture} 
           \draw[color=white] (0,0) circle (0.1cm); 
           \end{tikzpicture}\\
           \begin{tikzpicture} 
           \filldraw[color=blue] (0,0) circle (0.1cm); 
           \end{tikzpicture}  &
           \begin{tikzpicture} 
           \draw[color=blue] (0,0) circle (0.1cm); 
           \end{tikzpicture}\\
           \begin{tikzpicture} 
           \filldraw[color=blue] (0,0) circle (0.1cm); 
           \end{tikzpicture} & 
           \begin{tikzpicture} 
           \draw[color=blue] (0,0) circle (0.1cm); 
           \end{tikzpicture}\\
           \hline
        \end{tabular}
    \end{center}
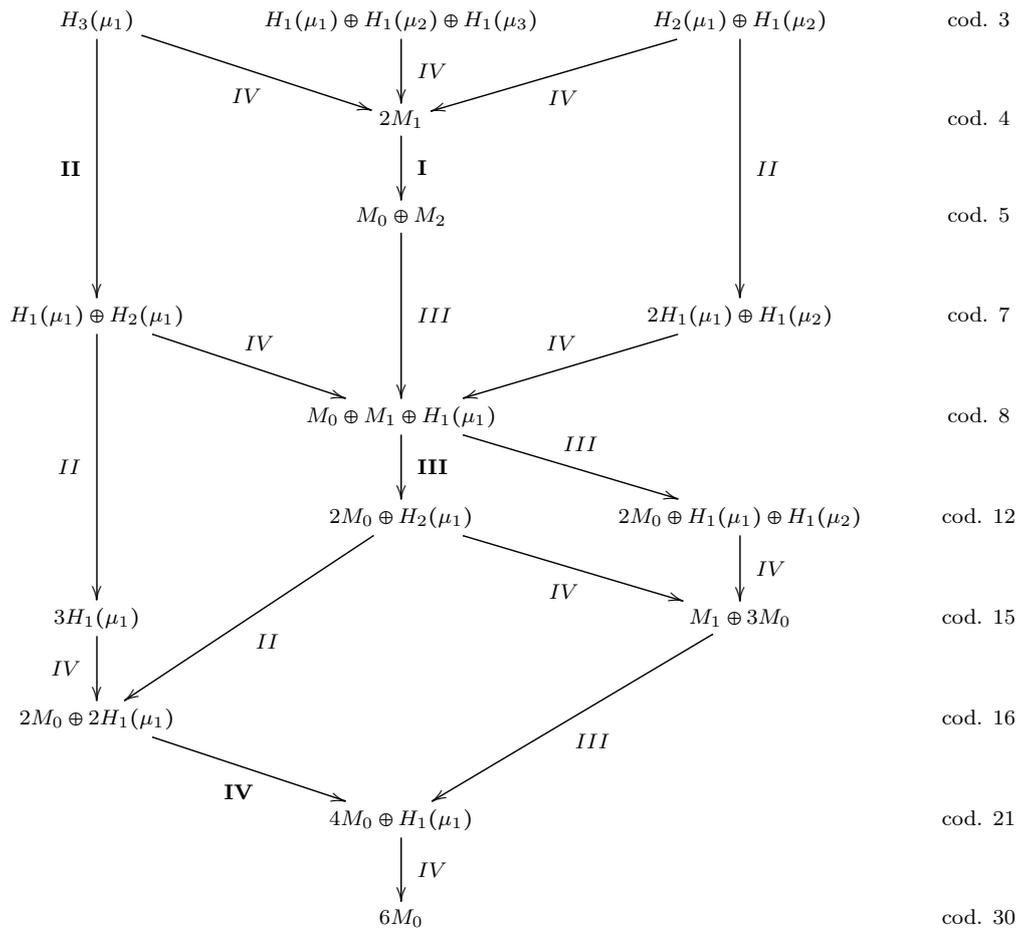
\begin{figure}
    \centering
    \begin{equation*}\label{g6cover-orbit}
    \text{\scriptsize{
    $\begin{split}
    \xymatrix{
    {H_3(\mu_1)} \ar[ddd]_-*+{\textbf{II}} \ar[dr]_-*+{IV} & { H_1(\mu_1) \oplus H_1(\mu_2) \oplus H_1(\mu_3)} \ar[d]^-*+{IV} & { H_2(\mu_1) \oplus H_1(\mu_2)} \ar[ddd]^-*+{II} \ar[dl]^-*+{IV} & {\text{cod. }3} \\
    &{2M_1} \ar[d]^-*+{\textbf{I}} & & {\text{cod. }4}\\
    &{M_0 \oplus M_2} \ar[dd]^-*+{III}  & & {\text{cod. }5}\\
    {H_1(\mu_1) \oplus  H_2(\mu_1) }\ar[ddd]_-*+{II} \ar[dr]^-*+{IV} & & {2H_1(\mu_1) \oplus  H_1(\mu_2)} \ar[dl]_-*+{IV} & {\text{cod. }7}\\
    & {  M_0 \oplus M_1 \oplus H_1(\mu_1) } \ar[d]^-*+{\textbf{III}}  \ar[dr]^-*+{III} & & {\text{cod. }8} \\
    & {  2M_0 \oplus H_2(\mu_1) } \ar[dr]_-*+{IV} \ar[ddl]^-*+{II} & {2M_0 \oplus H_1(\mu_1) \oplus H_1(\mu_2)} \ar[d]^-*+{IV} & {\text{cod. }12}\\
    {3H_1(\mu_1)}\ar[d]_-*+{IV} & & {M_1 \oplus 3M_0} \ar[ddl]^-*+{III} & {\text{cod. }15}\\
    {2M_0 \oplus 2H_1(\mu_1)} \ar[dr]_-*+{\textbf{IV}} & & & {\text{cod. }16} \\
    & {4M_0 \oplus H_1(\mu_1)} \ar[d]^-*+{IV} & & {\text{cod. }21}\\
    & {6M_0} & & {\text{cod. }30}
    }
    \end{split}$
    }}
    \end{equation*}
    \caption{Orbit stratification for $6\times 6$ skew-symmetric matrix pencils along with the codimensions. Rules from Theorem \ref{congruence} used for each structure transition are indicated near the corresponding edge.}
    \label{fig:Orbit_Ex}
\end{figure}
\end{example}
\begin{example} \label{example2}
All the orbits are considered under congruence transformation in this example. In Figure \ref{fig:Orbit_Ex_7}, we draw the orbit stratification graph of $7\times7$ skew-symmetric matrix pencil. {{\it{Rules I-IV}}} corresponding to each edge of the graph are mentioned, as well as the codimensions of each orbit.
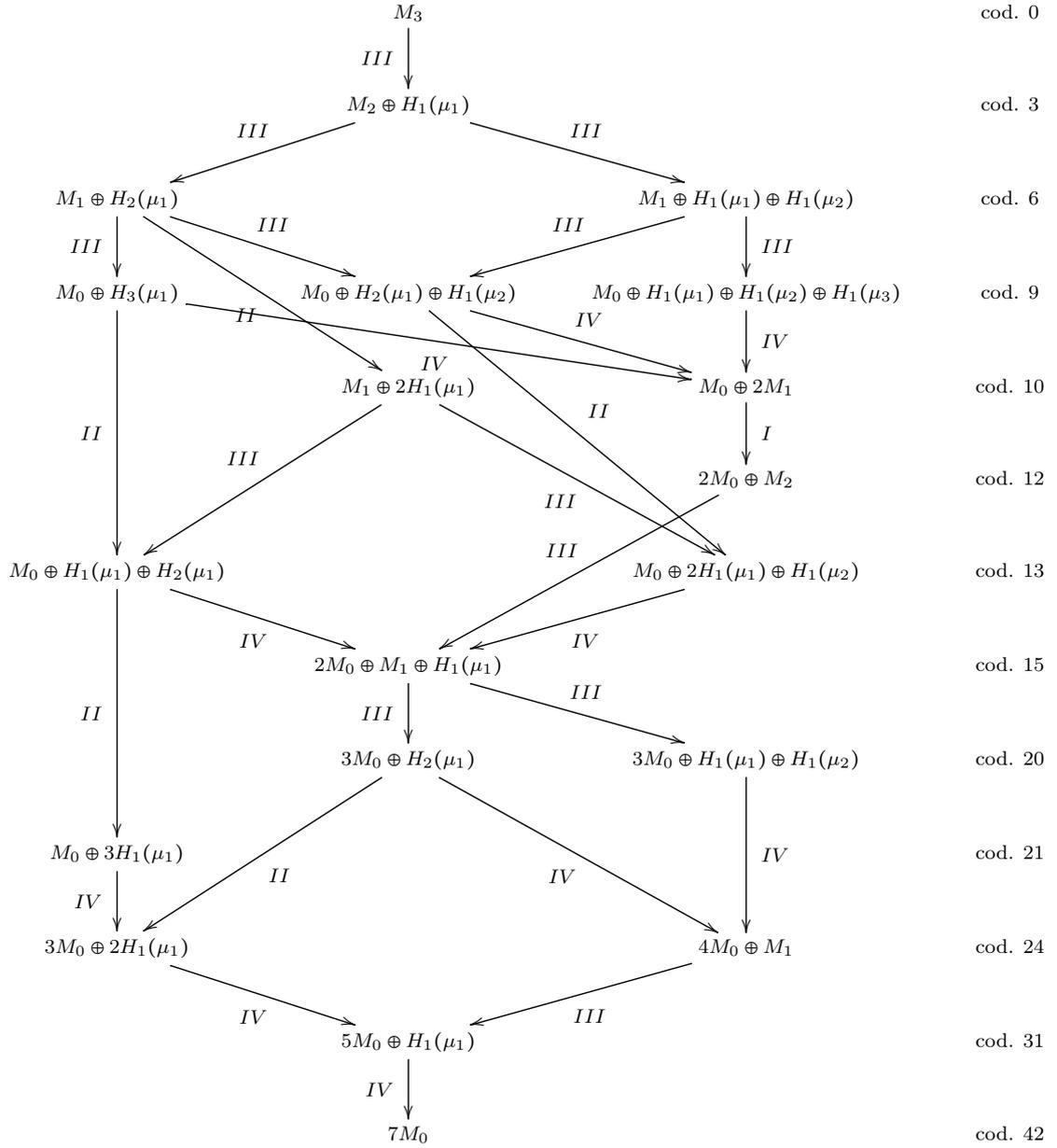
\begin{figure}
    \centering
    \begin{equation*}\label{g7cover-orbit}
    \text{\scriptsize{
    $\begin{split}
    \xymatrix{
     & {M_3} \ar[d]_-*+{III} & &  {\text{cod. 0}}\\
     & {M_2 \oplus H_1(\mu_1)} \ar[dr]^-*+{III} \ar[dl]_-*+{III} & &  {\text{cod. 3}}\\
     {M_1 \oplus H_2(\mu_1)} \ar[dr]^-*+{III} \ar[ddr]_-*+{II} \ar[d]_-*+{III} &  & {M_1 \oplus H_1(\mu_1) \oplus H_1(\mu_2)} \ar[dl]_-*+{III} \ar[d]^-*+{III} &{\text{cod. 6}} \\
     {M_0 \oplus H_3(\mu_1)} \ar[drr]_-*+{IV} \ar[ddd]_-*+{II} & {M_0 \oplus H_2(\mu_1) \oplus H_1(\mu_2)} \ar[dr]^-*+{IV} \ar[dddr]^-*+{II} & {M_0 \oplus H_1(\mu_1) \oplus H_1(\mu_2) \oplus H_1(\mu_3)} \ar[d]^-*+{IV}  & {\text{cod. 9}} \\
      & {M_1 \oplus 2H_1(\mu_1)} \ar[ddl]_-*+{III} \ar[ddr]_-*+{III} &  {M_0 \oplus 2M_1} \ar[d]^-*+{I} &   {\text{cod. 10}}\\
     &   & {2M_0 \oplus M_2} \ar[ddl]_-*+{III} &  {\text{cod. 12}}\\
     {M_0 \oplus H_1(\mu_1) \oplus H_2(\mu_1)} \ar[dr]_-*+{IV} \ar[ddd]_-*+{II} & {}  & {M_0 \oplus 2H_1(\mu_1) \oplus H_1(\mu_2)}  \ar[dl]^-*+{IV}  & {\text{cod. 13}} \\
     & {2M_0 \oplus M_1 \oplus H_1(\mu_1)} \ar[d]_-*+{III} \ar[dr]^-*+{III} & {}  &  {\text{cod. 15}}\\
     & {3M_0 \oplus H_2(\mu_1)} \ar[ddl]^-*+{II} \ar[ddr]_-*+{IV} & {3M_0 \oplus H_1(\mu_1) \oplus H_1(\mu_2)} \ar[dd]^-*+{IV} &  {\text{cod. 20}}\\
     {M_0 \oplus 3H_1(\mu_1)} \ar[d]_-*+{IV} & {}  & &   {\text{cod. 21}}\\
      {3M_0 \oplus 2H_1(\mu_1)} \ar[dr]_-*+{IV} & &  {4M_0 \oplus M_1} \ar[dl]^-*+{III} &   {\text{cod. 24}}\\
     & {5M_0 \oplus H_1(\mu_1)} \ar[d]_-*+{IV} & &  {\text{cod. 31}}\\
     & {7M_0} & &  {}   {\text{cod. 42}}
    }
    \end{split}$
    }}
    \end{equation*}
    \caption{Orbit stratification for $7\times 7$ skew-symmetric matrix pencil along with the codimensions. Rules from Theorem \ref{congruence} used for each structure transition are indicated near the corresponding edge.}
    \label{fig:Orbit_Ex_7}
\end{figure}
\end{example}
\section{Relation between closure hierarchy graphs of different dimensions}
In this section, we discuss a result that helps us to draw the closure hierarchy graph of size $(n+k)\times (n+k)$ from that of $n \times n$ and vice-versa. Note that the method allows us to only draw a part of the graph which can then be completed using the rules in Theorems  \ref{congruence} and \ref{co-for_go_up}. Also, note that two orbits having the same codimension in one graph might not have the same codimension in the new graph. 
\begin{theorem}\label{hack_graph}
    Let $P_1$ and $P_2$ be two $n \times n$ skew-symmetric matrix pencils. Then, two $(n+1) \times (n+1)$ skew-symmetric matrix pencils $Q_i$ are constructed from $P_i$, such that, under congruence, $Q_i$ is congruent to $P_i \oplus M_0$, for $i \in \{1,2\}$. $\orb_{P_1}^c$ covers $\orb_{P_2}^c$ if and only if $\orb_{Q_1}^c$ covers $\orb_{Q_2}^c$.
\end{theorem}
\begin{proof}
    In Theorem \ref{congruence}, {\it Rule 1}: $M_{j} \oplus M_{k} \rightsquigarrow M_{j-1} \oplus M_{k+1}$, $k - j = \min\big\{\{k_1-j,k_2-j,\dots\}\cup \{k-j_1,k-j_2,\dots\}\big\}$, $j_t\le k$ and $j \le k_u $, and {\it Rule 3}: $M_{j+1} \oplus H_{k}(\mu) \rightsquigarrow M_{j} \oplus H_{k+1}(\mu)$, $j,k=0,1,2,\dots$, $\mu \in \overline{\mathbb C}$, do not effect the new summand $M_0$ in ${P_{i}}\oplus M_0$. Thus, ${P_{1}}$ $\rightsquigarrow$ ${P_{2}}$ is equivalent to ${P_{1}}$$\oplus M_0 \rightsquigarrow$ ${P_{2}}\oplus M_0$, for any {\it Rule i}, $i \in \{1,2,3,4\}$. Hence, the result follows.
\end{proof}
The above theorem can be illustrated by comparing the graphs in Figures~\ref{fig:Orbit_Ex} and ~\ref{fig:Orbit_Ex_7}. Using arguments similar to Theorem \ref{hack_graph}, analogous result for matrix pencils under the equivalence transformation can be stated.
\section{Result and future work}
In this manuscript, we have presented a necessary and sufficient condition for one orbit of a skew-symmetric matrix pencil to cover the orbit of another skew-symmetric matrix pencil. The desired condition is presented in terms of structure transitions for canonical forms and coin moves. We conclude our paper with a technique to manually draw the whole or a part of the stratification graph from a given stratification graph for pencils of different dimensions. 
As part of our future research we plan to develop a characterization for closures of bundles of skew-symmetric pencils (similarly, as it is developed in \cite{DeDo22} for general matrix pencils, see also \cite{Dmyt23}) as well as to obtain a cover relation for bundles of skew-symmetric matrix pencils. Extending these results to skew-symmetric matrix polynomials and implementing all the results in the Stratigraph software \cite{Stra24} are other directions for future research.    

{\footnotesize 
\bibliographystyle{plain}
\bibliography{library}
}

\end{document}